\documentclass[12pt]{article}%
\usepackage{graphicx}
\usepackage[intlimits]{amsmath}
\usepackage{latexsym}
\usepackage{amsfonts}
\usepackage{amssymb}%
\makeatletter
\newfont{\footsc}{cmcsc10 at 8truept}
\newfont{\footbf}{cmbx10 at 8truept}
\newfont{\footrm}{cmr10 at 10truept}
\newtheorem{theorem}{Theorem}

\newtheorem{conjecture}[theorem]{Conjecture}
\newtheorem{corollary}[theorem]{Corollary}

\newtheorem{lemma}[theorem]{Lemma}

\newenvironment{proof}[1][Proof]{\noindent{\textbf {#1}  }}  {\hfill$\Box$\bigskip}

\begin{document}

\title{\textbf{Maxima of the }$Q$\textbf{-index: forbidden even cycles} }
\author{Vladimir Nikiforov\thanks{Department of Mathematical Sciences, University of
Memphis, Memphis TN 38152, USA; \textit{email: vnikiforv@memphis.edu}} \ and
Xiying Yuan\thanks{Corresponding author. Department of Mathematics, Shanghai
University, Shanghai, 200444, China; email:
\textit{xiyingyuan2007@hotmail.com}} \thanks{Research supported by National
Science Foundation of China (No. 11101263), and by a grant of
\textquotedblleft The First-class Discipline of Universities in
Shanghai\textquotedblright.} }
\maketitle

\begin{abstract}
Let $G$ be a graph of order $n$ and let $q\left(  G\right)  $ be the largest
eigenvalue of the signless Laplacian of $G$. Let $S_{n,k}$ be the graph
obtained by joining each vertex of a complete graph of order $k$ to each
vertex of an independent set of order $n-k;$ and let $S_{n,k}^{+}$ be the
graph obtained by adding an edge to $S_{n,k}.$

It is shown that if $k\geq2,$ $n\geq400k^{2},$ and $G$ is a graph of order
$n,$ with no cycle of length $2k+2,$ then $q\left(  G\right)  <q\left(
S_{n,k}^{+}\right)  ,$ unless $G=S_{n,k}^{+}.$\ This result completes the
proof of a conjecture of de Freitas, Nikiforov and Patuzzi.\smallskip



\end{abstract}

\section{Introduction}

Given a graph $G,$ the $Q$-index of $G$ is the largest eigenvalue $q\left(
G\right)  $ of its signless Laplacian $Q\left(  G\right)  $. In this paper we
study how large can $q\left(  G\right)  $ be if $G$ is a graph of given order
and contains no cycle of given even length.

Thus, let $S_{n,k}$ be the graph obtained by joining each vertex of a complete
graph of order $k$ to each vertex of an independent set of order $n-k,$ that
is to say, $S_{n,k}=K_{k}\vee\overline{K}_{n-k}.$ Also, let $S_{n,k}^{+}$ be
the graph obtained by adding an edge to $S_{n,k}.$ In \cite{FNP13} the
following conjecture has been raised:

\begin{conjecture}
\label{con1} Let $k\geq2$ and let $G$ be a graph of sufficiently large order
$n.$ If $G$ has no cycle of length $2k+1,$ then $q\left(  G\right)  <q\left(
S_{n,k}\right)  ,$ unless $G=S_{n,k}.$ If $G$ has no cycle of length $2k+2,$
then $q\left(  G\right)  <q\left(  S_{n,k}^{+}\right)  ,$ unless
$G=S_{n,k}^{+}.$
\end{conjecture}

In \cite{Nik13} it was shown that Conjecture \ref{con1} is asymptotically true
and some proof technique has been outlined. In \cite{Yua14} and \cite{Yua14a}
the second author developed this technique further and succeeded to prove
Conjecture \ref{con1} for forbidden odd cycles. In this paper we shall prove
the remaining case of the conjecture, which turns out to be by far more
difficult than the odd case. Thus, our main result is the following theorem:

\begin{theorem}
\label{mt} Let $k\geq2,$ $n\geq400k^{2},$ and let $G$ be a graph of order $n$.
If $G$ has no cycle of length $2k+2,$ then $q\left(  G\right)  <q\left(
S_{n,k}^{+}\right)  ,$ unless $G=S_{n,k}^{+}.$
\end{theorem}

In the next section we prepare the ground for the proof of Theorem \ref{mt}
and in Section \ref{pf} we give the proof itself. At the end we give a sum up
of our work.

\section{Notation and supporting results}

For graph notation and concepts undefined here, we refer the reader to
\cite{Bol98}. For introductory material on the signless Laplacian see the
survey of Cvetkovi\'{c} \cite{Cve10} and its references. In particular, let
$G$ be a graph, and $X$ and $Y$ be disjoint sets of vertices of $G.$ We write:

- $V\left(  G\right)  $ for the set of vertices of $G$ and $E\left(  G\right)
$ for the set of edges of $G;$

- $\left\vert G\right\vert $ for the number of vertices of $G$ and $e\left(
G\right)  $ for the number of edges of $G$;

- $G\left[  X\right]  $ for the graph induced by $X,$ and $e\left(  X\right)
$ for $e\left(  G\left[  X\right]  \right)  ;$

- $G_{u}$ for the graph induced by the set $V\left(  G\right)  \backslash
\left\{  u\right\}  ,$ where $u\in V\left(  G\right)  ;$

- $e\left(  X,Y\right)  $ for the number of edges joining vertices in $X$ to
vertices in $Y;$

- $\Gamma_{G}\left(  u\right)  $ (or simply $\Gamma\left(  u\right)  $) for
the set of neighbors of a vertex $u,$ and $d_{G}\left(  u\right)  $ (or simply
$d\left(  u\right)  $) for $\left\vert \Gamma\left(  u\right)  \right\vert
.\medskip$

We write $P_{k},$ $C_{k},$ and $K_{k}$ for the path, cycle, and complete graph
of order $k.$ We write $F\subset G$ to indicate that $F$ is a subgraph of $G,$
and we say that a graph $G$ is $F$\emph{-free} if $G$ contains no subgraphs
isomorphic to $F.$

\subsection{Some auxiliary results}

Here we state several known results, all of which are used in Section
\ref{pf}. We start with a result of Dirac \cite{Dir52}.

\begin{theorem}
\label{Dit} If $G$ is a graph with $\delta\left(  G\right)  \geq2,$ then $G$
contains a cycle longer than $\delta\left(  G\right)  .$
\end{theorem}

Dirac's result has been further developed by Erd\H{o}s and Gallai
\cite{ErGa59}. We shall need the following classical theorems from their paper.

\begin{theorem}
\label{EGp}Let $k\geq1.$ If $G$ is a graph of order $n,$ with no $P_{k+2},$
then $e\left(  G\right)  \leq kn/2,$ with equality holding if and only if $G$
is a union of disjoint copies of $K_{k+1}.$
\end{theorem}

\begin{theorem}
\label{EGc}Let $k\geq2.$ If $G$ is a graph of order $n,$ with no cycle longer
than $k,$ then $e\left(  G\right)  \leq k\left(  n-1\right)  /2,$ with
equality holding if and only if $G$ is a union of copies of $K_{k},$ all
sharing a single vertex.
\end{theorem}

The following structural extension of Theorem \ref{EGp} has been established
in \cite{Nik09}.

\begin{lemma}
\label{Nl} Let $k\geq1$ and let the vertices of a graph $G$ be partitioned
into two sets $A$ and $B$. If
\[
2e\left(  A\right)  +e\left(  A,B\right)  >\left(  2k-1\right)  \left\vert
A\right\vert +k\left\vert B\right\vert ,
\]
then there exists a path of order $2k+1$ with both endvertices in $A.$
\end{lemma}

To state the next result define the graph $L_{t,k}$ by $L_{t,k}:=K_{1}\vee
tK_{k},$ i.e., $L_{t,k}$ consists of $t$ complete graphs $K_{k+1},$ all
sharing a single common vertex. In \cite{NiYu13}, the following stability
result has been proved.

\begin{theorem}
\label{AS+}Let $k\geq2,$ $n\geq2k+3,$ and $G$ be a graph of order $n$ with
$\delta\left(  G\right)  \geq k.$ If $G$ is connected, then $P_{2k+3}\subset
G,$ unless one of the following holds:

\emph{(i)} $G\subset S_{n,k}^{+};$

\emph{(ii)} $n=tk+1$ and $G=L_{t,k};$

\emph{(iii)} $n=tk+2$ and $G\subset K_{1}\vee(\left(  t-1\right)  K_{k}\cup
K_{k+1});$

\emph{(iv) }$n=\left(  s+t\right)  k+2$ and $G\ $is obtained by joining the
centers of two disjoint graphs $L_{s,k}$ and $L_{t,k}$.
\end{theorem}

We shall need, in fact, a particular corollary of Theorem \ref{AS+}, which is
easy to check directly.

\begin{corollary}
\label{cAS+}Let $k\geq2,$ $n\geq2k+3,$ and let $G$ be a connected $P_{2k+3}%
$-free graph of order $n$ with $\delta\left(  G\right)  \geq k.$ Then
$e\left(  G\right)  \leq\left(  k+1\right)  n/2,$ unless $G\subset S_{n,k}%
^{+}$ and $e\left(  G\right)  \leq kn.$
\end{corollary}

Another statement that we shall need is a variant of Theorem \ref{AS+} for the
case $k=1;$ we omit its easy proof.

\begin{lemma}
\label{AS+1}If $G$ is a connected graph of order $n\geq5.$ If $G$ contains no
$P_{5}$, then one of the following holds:

\emph{(i)} $G\subset S_{n,1}^{+};$

\emph{(ii) }$G\ $is obtained by joining the centers of two disjoint stars.
\end{lemma}

We finish this subsection with two known upper bounds on $q\left(  G\right)
.$ The first one can be traced back to Merris \cite{Mer98}, while the case of
equality has been established in \cite{FeYu09}.

\begin{theorem}
\label{tM}For every graph $G,$
\begin{equation}
q\left(  G\right)  \leq\max\left\{  d\left(  u\right)  +\frac{1}{d\left(
u\right)  }\sum_{v\in\Gamma\left(  u\right)  }d\left(  v\right)  :u\in
V\left(  G\right)  \right\}  . \label{Mb}%
\end{equation}
If $G$ is connected, equality holds if and only if $G$ is regular or
semiregular bipartite.
\end{theorem}

Finally, let us mention the following bound, due to Das \cite{Das04}.

\begin{theorem}
\label{tD}If $G$ is a graph with $n$ vertices and $m$ edges, then
\begin{equation}
q\left(  G\right)  \leq\frac{2m}{n-1}+n-2, \label{Db}%
\end{equation}
with equality holding if and only if $G$ is either complete, or is a star, or
is a complete graph with one isolated vertex.
\end{theorem}

\section{\label{pf}Proof of Theorem \ref{mt}}

Before going further, we shall make three remarks. First, recall an estimate
of $q\left(  S_{n,k}^{+}\right)  $ given in \cite{FNP13}, where it was shown
that if $k\geq2$ and $n>5k^{2},$ then
\begin{equation}
n+2k-2-\frac{2k\left(  k-1\right)  }{n+2k+2}>q\left(  S_{n,k}^{+}\right)
>n+2k-2-\frac{2k\left(  k-1\right)  }{n+2k-3}. \label{b0}%
\end{equation}

Second, note that if $G$ is a graph with $q(G)\geq q\left(  S_{n,k}%
^{+}\right)  ,$ then $e(G)$ cannot be much smaller than $e\left(  S_{n,k}%
^{+}\right)  .$ Indeed, in view of Das's bound (\ref{Db}) we have
\[
q\left(  S_{n,k}^{+}\right)  \leq q(G)\leq\frac{2e(G)}{n-1}+n-2,
\]
and so
\begin{equation}
e(G)\geq kn-k^{2}+1=e\left(  S_{n,k}^{+}\right)  -\frac{k\left(  k-1\right)
}{2}. \label{b1}%
\end{equation}

Finally, given a vertex $u$ of graph $G,$ note that
\[
\sum_{\left\{  u,v\right\}  \in E\left(  G\right)  }d\left(  v\right)
=2e\left(  \Gamma\left(  u\right)  \right)  +e\left(  \Gamma\left(  u\right)
,V\left(  G\right)  \backslash\Gamma\left(  u\right)  \right)  .
\]
We shall use this equality with no explicit reference.

Our proof of Theorem \ref{mt} is rather long and complicated. To improve the
presentation we have extracted a large segment of it into Lemma \ref{l1} and
Theorem \ref{mt1} below. Since these assertions are subordinate to the proof
of Theorem \ref{mt}, their statements look somewhat technical.

\begin{lemma}
\label{l1}Let $k\geq2,$ $n\geq400k^{2},$ let $G$ be a $C_{2k+2}$-free graph of
order $n$ with%
\[
q\left(  G\right)  \geq q\left(  S_{n,k}^{+}\right)  .
\]
Let $w$ be a dominating vertex of $G,$ suppose that $G_{1},\ldots,G_{p}$ are
the components of $G_{w}$ of order at most $3k^{2},$ and let $H:=G-\cup
_{i=1}^{p}G_{i}.$ Then $\left\vert H\right\vert \geq3n/10$ and \textbf{ }
\[
q\left(  H\right)  >q\left(  S_{\left\vert H\right\vert ,k}^{+}\right)  ,
\]
unless $H=G$ and $q\left(  H\right)  =q\left(  G\right)  =q\left(  S_{n,k}%
^{+}\right)  .$
\end{lemma}

\begin{proof}
If $p=0$, the proof is completed, so let us assume that $p\geq1.$ We shall use
induction on $p.$ Let $\left\vert G_{i}\right\vert :=n_{i},$ and
$H_{s}:=G-\cup_{i=1}^{s}G_{i}.$ First we shall prove that $q\left(
H_{1}\right)  >q\left(  S_{n-n_{1},k}^{+}\right)  .$ For short, set
$q:=q\left(  G\right)  $ and $V:=V\left(  G\right)  .$ Letting $\left(
x_{1},\ldots,x_{n}\right)  $ be a positive unit eigenvector to $q,$ from the
eigenequation for $q$ and the vertex $w$ we see that%
\[
\left(  q-n+1\right)  x_{w}=\sum_{i\in V\backslash\left\{  w\right\}  }%
x_{i}\leq\sqrt{\left(  n-1\right)  \left(  1-x_{w}^{2}\right)  }.
\]
Now, in view of $n\geq400k^{2}\geq5k^{2}$ and (\ref{b0}), we see that
\[
q\geq q\left(  S_{n,k}^{+}\right)  >n+2k-3,
\]
and so,
\begin{equation}
x_{w}^{2}\leq\frac{n-1}{\left(  q-n+1\right)  ^{2}+n-1}\leq\frac
{n-1}{n-1+4\left(  k-1\right)  ^{2}}<1-\frac{4\left(  k-1\right)  ^{2}%
}{n+4k^{2}}. \label{in4}%
\end{equation}

Next, letting
\begin{equation}
x=\max\left\{  x_{i}\text{
$\vert$
}i\in V\left(  G_{1}\right)  \right\}  , \label{b2}%
\end{equation}
we have
\[
qx\leq n_{1}x+\left(  n_{1}-1\right)  x+x_{w},
\]
and so%
\[
x\leq\frac{x_{w}}{q-2n_{1}+1}\leq\frac{x_{w}}{n+2k-2n_{1}-2}\leq\frac{x_{w}%
}{n-6k^{2}}.
\]
Note that%
\begin{align*}
q\left(  H_{1}\right)   &  \geq q\left(  H_{1}\right)  \left(  1-\sum_{i\in
V\left(  G_{1}\right)  }x_{i}^{2}\right)  \geq\sum_{\left\{  i,j\right\}  \in
E\left(  H_{1}\right)  }\left(  x_{i}+x_{j}\right)  ^{2}\\
&  =q\left(  G\right)  -\sum_{\left\{  i,j\right\}  \in E\left(  G_{1}\right)
}\left(  x_{i}+x_{j}\right)  ^{2}-\sum_{i\in V\left(  G_{1}\right)  }\left(
x_{i}+x_{w}\right)  ^{2}.
\end{align*}
On the other hand, using (\ref{b2}) and (\ref{in4}), we see that
\begin{align*}
\sum_{\left\{  i,j\right\}  \in E\left(  G_{1}\right)  }\left(  x_{i}%
+x_{j}\right)  ^{2}+\sum_{i\in V\left(  G_{1}\right)  }\left(  x_{i}%
+x_{w}\right)  ^{2}  &  \leq2n_{1}\left(  n_{1}-1\right)  x^{2}+n_{1}\left(
x+x_{w}\right)  ^{2}\\
&  \leq n_{1}\left(  1+\frac{2}{n-6k^{2}}+\frac{2n_{1}-1}{\left(
n-6k^{2}\right)  ^{2}}\right)  x_{w}^{2}\\
&  \leq n_{1}\left(  1+\frac{21}{10\left(  n-6k^{2}\right)  }\right)  \left(
1-\frac{4\left(  k-1\right)  ^{2}}{n+4k^{2}}\right) \\
&  \leq n_{1}\left(  1-\frac{3}{2n}\right)  .
\end{align*}
Hence, (\ref{b0}) implies that
\begin{align*}
q\left(  H_{1}\right)   &  \geq q\left(  G\right)  -n_{1}+\frac{3n_{1}}%
{2n}\geq n-n_{1}+2k-2-\frac{2k\left(  k-1\right)  }{n+2k-3}+\frac{3n_{1}}%
{2n}\\
&  >n-n_{1}+2k-2-\frac{2k\left(  k-1\right)  }{n-n_{1}+2k+2}>q\left(
S_{n-n_{1},k}^{+}\right)  .
\end{align*}
We can apply this argument repeatedly, as long as the remaining graph $H_{i}$
is of order at least $100k^{2}.$ Assume that there exists some $j$ such that
$\left\vert H_{j-1}\right\vert \geq100k^{2},$ while $\left\vert H_{j}%
\right\vert <100k^{2}.$ Set for short $t:=n_{1}+\cdots+n_{j},$ and note that
\[
q\left(  H_{j}\right)  \geq q\left(  G\right)  -t+\frac{3t}{2n}.
\]
Since $C_{s}\nsubseteq H_{j}$ for any $s\geq2k+2,$ Theorem \ref{EGc} implies
that
\[
e\left(  H_{j}\right)  \leq\left(  k+1/2\right)  \left(  n-t-1\right)
\]
and Das's bound (\ref{Db}) implies that
\[
q\left(  H_{j}\right)  \leq\frac{2e\left(  H_{j}\right)  }{n-t-1}+n-t-2\leq
n-t+2k-1.
\]
In view of
\[
q\left(  H_{j}\right)  \geq q\left(  G\right)  -t+\frac{3t}{2n}>n-t+2k-2-\frac
{2k\left(  k-1\right)  }{n+2k-3}+\frac{3t}{2n},
\]
we have
\[
t<\frac{7}{10}n,
\]
which implies that $\left\vert H\right\vert \geq3n/10>100k^{2},$ and this
contradiction completes the proof of Lemma \ref{l1}.
\end{proof}

Lemma \ref{l1} is crucial for the next theorem, where it will be used to
improve the structure of a graph at the price of a moderate reduction of its order.

\begin{theorem}
\label{mt1}Let $k\geq2,$ $n\geq400k^{2},$ and let $G$ be a graph of order $n$.
If $q\left(  G\right)  \geq q\left(  S_{n,k}^{+}\right)  $ and $\Delta\left(
G\right)  =n-1$, then $C_{2k+2}\subset G,$ unless $G=S_{n,k}^{+}.$
\end{theorem}

\begin{proof}
Assume that $C_{2k+2}\nsubseteq G.$ To prove the theorem we need to show that
$G=S_{n,k}^{+}.$ Let $w$ be a dominating vertex of $G.$ Applying Lemma
\ref{l1}, we can find an induced subgraph $H\subset G$ with the following properties:

- $H$ is $C_{2k+2}$-free;

- $w$ is a dominating vertex in $H$;

- $\left\vert H\right\vert \geq3n/10>100k^{2};$

- every component of $H_{w}$ is of order greater than $3k^{2};$

- $q\left(  H\right)  >q\left(  S_{\left\vert H\right\vert ,k}^{+}\right)  ,$
unless $H=G$ and $q\left(  H\right)  =q\left(  G\right)  =q\left(  S_{n,k}%
^{+}\right)  .$

Thus, to complete the proof of Theorem \ref{mt1} it is enough to prove the
following statement:\medskip

\textbf{Theorem A }\emph{Let }$k\geq2,$\emph{ }$n\geq100k^{2},$\emph{ let }%
$G$\emph{ be a }$C_{2k+2}$\emph{-free graph of order }$n,$\emph{ and let }%
$w$\emph{ be a dominating vertex of }$G$\emph{ such that each component of
}$G_{w}$\emph{ is greater than }$3k^{2}$\emph{. If }$q\left(  G\right)  \geq
q\left(  S_{n,k}^{+}\right)  ,$\emph{\ then }$G=S_{n,k}^{+}.\medskip$

We proceed with the proof of Theorem A, keeping the notation for $G$ and $n,$
although now $n>100k^{2}.$ Thus, assume that $G$ and $n$ satisfy the premises
of Theorem A and $q\left(  G\right)  \geq q\left(  S_{n,k}^{+}\right)  .$ We
shall prove that $G=S_{n,k}^{+}.$ For convenience choose $G$ so that it has
maximum $Q$-index among all graphs satisfying the premises of Theorem A. Next
observe that inequality (\ref{b1}) implies that
\[
e(G)\geq kn-k^{2}+1,
\]
and so
\begin{equation}
e(G_{w})=e(G)-n+1\geq\left(  k-1\right)  n-k^{2}+2. \label{be}%
\end{equation}
\medskip

We shall dispose of the case $k=2$ before anything else, as most of our
arguments work for $k\geq3$ and need changes to work for $k=2.$ \medskip

\textbf{Claim 1. }\emph{Theorem A holds for }$k=2.\medskip$

\emph{Proof. }If $k=2,$ then $G_{w}$ consist of components of order greater
than $12.$ Since $\delta\left(  G_{w}\right)  \geq1$ and $P_{5}\nsubseteq
G_{w},$ Lemma \ref{AS+1} implies that the components of $G_{w}$ are of the two
types given in clauses \emph{(i) }and\emph{ (ii)}. If $G_{i}\subset
S_{\left\vert G_{i}\right\vert ,1}^{+}$, we see that $G_{i}=S_{\left\vert
G_{i}\right\vert ,1}^{+},$ as $q\left(  G\right)  $ is maximal. Now assume
that $G_{i}$ is a component of $G_{w}$ consisting of two stars whose centers
are joined by an edge; let $u$ and $v$ be the centers of these stars; let
$u_{1},\ldots,u_{s}$ be the neighbors of $u$ and $v_{1},\ldots,v_{t}$ be the
neighbors of $v.$ Let $\left(  x_{1},\ldots,x_{n}\right)  $ be a unit positive
eigenvector to $q\left(  G\right)  ;$ by symmetry suppose that $x_{u}\geq
x_{v}.$ Remove all edges $\left\{  v_{i},v\right\}  $ and join $v_{1}%
,\ldots,v_{t}$ to $u.$ In this way $G_{i}$ is transformed into a star
$S_{s+t+2,1};$ now make it an $S_{s+t+2,1}^{+}$ by adding an edge to it. A
brief calculation shows that the resulting graph $G^{\prime}$ satisfies
$q\left(  G^{\prime}\right)  >q\left(  G\right)  ,$ which contradicts the
choice of $G$. Hence each component $G_{i}$ of $G_{w}$ satisfies
$G_{i}=S_{\left\vert G_{i}\right\vert ,1}^{+}$.

Finally, if $G_{w}$ has two components $G_{1}=S_{\left\vert G_{1}\right\vert
,1}^{+}$ and $G_{2}=S_{\left\vert G_{2}\right\vert ,1}^{+},$ replace them by a
component $S_{\left\vert G_{1}\right\vert +\left\vert G_{2}\right\vert ,1}%
^{+}$ and write $G^{\prime\prime}$ for the resulting graph. Clearly
$C_{6}\nsubseteq G^{\prime\prime}$ and we shall show that $q\left(
G^{\prime\prime}\right)  >q\left(  G\right)  ,$ which contradicts our choice
of $G$. Let $u\ $and $v$ be the dominating vertices in $G_{1}$ and $G_{2}.$
Let $\mathbf{x}=\left(  x_{1},\ldots,x_{n}\right)  $ be a positive unit
eigenvector to $q\left(  G\right)  .$ By symmetry we may suppose that
$x_{u}\geq x_{v}.$ Now remove all edges of $G_{2}$ and join all vertices of
$G_{2}$ to $u.$ In this way $G_{1}$ and $G_{2}$ are replaced by a single
component $S_{\left\vert G_{1}\right\vert +\left\vert G_{2}\right\vert ,1}%
^{+}.$ For short, let $n_{1}:=\left\vert G_{1}\right\vert ,$ $n_{2}%
:=\left\vert G_{2}\right\vert ,$ and $q:=q\left(  G\right)  .$ Let
$W:=V\left(  S_{n_{2},k-1}^{+}\right)  \backslash\left\{  v\right\}  $ and let
$v^{\prime},v^{\prime\prime}\in W$ be the two exceptional vertices of $G_{2}$
such that $\left\{  v^{\prime},v^{\prime\prime}\right\}  \in E\left(
G\right)  .$ By symmetry, $x_{v^{\prime}}=x_{v^{\prime\prime}}$ and from the
eigenequations for $q$ we see that
\begin{align*}
qx_{v^{\prime}}  &  =3x_{v^{\prime}}+x_{v^{\prime\prime}}+x_{v}+x_{w}%
=4x_{v^{\prime}}+x_{v}+x_{w},\\
qx_{v}  &  =n_{2}x_{v}+x_{v^{\prime}}+x_{v^{\prime\prime}}+x_{w}+\sum_{s\in
W\backslash\left\{  v^{\prime},v^{\prime\prime}\right\}  }x_{s}\\
&  >n_{2}x_{v}+2x_{v^{\prime}}+x_{w}.
\end{align*}
Excluding $x_{w}$ from these relations, after some algebra we see that
\[
\left(  q-n_{2}+1\right)  x_{v}>\left(  q-2\right)  x_{v^{\prime}}%
\]
and so $x_{v^{\prime\prime}}=x_{v^{\prime}}<x_{v}.$ A brief calculation shows
that the resulting graph $G^{\prime}$ satisfies $q\left(  G^{\prime}\right)
>q\left(  G\right)  ,$ which contradicts the choice of $G$. This contradiction
shows that $G_{w}$ has only one component and so $G=S_{n,2}^{+},$ completing
the proof of Claim 1.\medskip

To the end of the proof we shall assume that $k\geq3.\medskip$

\textbf{Claim 2. }\emph{There exists an induced subgraph }$H$\emph{ of }%
$G_{w}$\emph{ such that }$\delta\left(  H\right)  \geq k-1$\emph{ and
}$\left\vert H\right\vert \geq n-k^{2}+k$\emph{.}\medskip

\emph{Proof. }Define a sequence of graphs, $F_{0}\supset F_{1}\supset
\cdots\supset F_{r}\supset\cdots$ using the following procedure:

$F_{0}:=G_{w};$

$i:=0$;

\textbf{while} $\delta(F_{i})<k-1$ \textbf{do begin}

\qquad select a vertex $v\in V(F_{i})$ with $d(v)=\delta(F_{i});$

\qquad$F_{i+1}:=F_{i}-v;$

$\qquad i:=i+1;$

\textbf{end.}

Note that for each $r=0,1,\ldots,$ we have $\left\vert F_{r}\right\vert
=n-r-1$ and $P_{2k+1}\nsubseteq F_{r};$ thus Theorem \ref{EGp} implies that
\[
e\left(  F_{r}\right)  \leq\left(  k-\frac{1}{2}\right)  \left(  n-r-1\right)
.
\]

On the other hand, in view of (\ref{be}), we find that
\begin{align}
e\left(  F_{r}\right)   &  =e\left(  G_{w}\right)  -\sum_{i=0}^{r-1}%
\delta(F_{i})\geq e\left(  G_{w}\right)  -r\left(  k-2\right) \nonumber\\
&  \geq\left(  k-1\right)  n-k^{2}+2-r\left(  k-2\right)  . \label{lb}%
\end{align}
Hence,
\[
\left(  k-1\right)  n-k^{2}+2-r\left(  k-2\right)  \leq\left(  k-\frac{1}%
{2}\right)  \left(  n-r-1\right)  ,
\]
and after some algebra we find that%
\[
3r\leq n+2k^{2}-2k-3<2n.
\]
that is to say, the procedure stops before $i\geq2n/3$. Next, with a more
involved argument, we shall show that the procedure stops before $i>k^{2}-k-1$.

Let $H=F_{r},$ where $r$ is the last value of the variable $i.$ Let $H_{i}$ be
a component of $H$ and set $n_{i}:=\left\vert H_{i}\right\vert .$ We claim
that $e\left(  H_{i}\right)  \leq\left(  k-1\right)  n_{i}.$

Indeed if $n_{i}\leq2k-1,$ then
\[
e\left(  H_{i}\right)  \leq\frac{n_{i}\left(  n_{i}-1\right)  }{2}\leq\left(
k-1\right)  n_{i}.
\]

If $n_{i}=2k$ and $H_{i}$ is Hamiltonian, then $H_{i}$ is a component of
$G_{w}$ as otherwise $P_{2k+1}\subset G_{w}.$ But all components of $G_{w}$
are of order at least $3k^{2}>2k,$ so $H_{i}$ is not Hamiltonian. In this
case, Ore's theorem \cite{Ore60} implies that
\[
e\left(  H_{i}\right)  \leq\frac{\left(  n_{i}-1\right)  \left(
n_{i}-2\right)  }{2}+1\leq\left(  k-1\right)  n_{i}.
\]

If $n\geq2k+1,$ in view of $P_{2k+1}\nsubseteq$ $H_{i}$ and $\delta\left(
H_{i}\right)  \geq k-1,$ it follows that $H_{i}$ satisfies one of the clauses
of Theorem \ref{AS+}, and so Corollary \ref{cAS+} implies that $e\left(
H_{i}\right)  \leq\left(  k-1\right)  n_{i}.$ Summing over all components of
$H,$ we find that%
\[
e\left(  H\right)  \leq\left(  k-1\right)  \left(  n-r-1\right)  .
\]
On the other hand, in view of (\ref{lb}) we find that
\[
\left(  k-1\right)  n-k^{2}+2-r\left(  k-2\right)  \leq\left(  k-1\right)
\left(  n-r-1\right)  ,
\]
and so, $r\leq k^{2}-k-1.$ Therefore, $H$ satisfies the requirements of Claim
2, which is thus proved.\medskip

Let $H^{\prime}$ be the subgraph of $G_{w}$ induced by the vertex set
$V\left(  G_{w}\right)  \backslash V\left(  H\right)  ,$ which may be empty.
Let $H_{1},\ldots,H_{p}$ be the components of $H$ and let $n_{1},\ldots,n_{p}$
be their orders.\medskip

\textbf{Claim 3. }\emph{Each component }$H_{i}$\emph{ of\ }$H$\emph{ satisfies
}$H_{i}\subset S_{\left\vert H_{i}\right\vert ,k-1}^{+}.\medskip$

\emph{Proof. }First note that each component of $G_{w}$ contains at most one
component of $H.$ Indeed, since for each component $H_{i}$ of $H$, we have
$\delta\left(  H_{i}\right)  \geq k-1\geq2,$ Dirac's Theorem \ref{Dit} implies
that $C_{l}\subset H_{i},$ for some $l\geq k;$ hence each component of $G_{w}$
contains at most one component of $H,$ as otherwise $P_{2k+1}\subset G_{w}.$

Further, if $H_{i}$ is a component of $H$ and $G_{i}$ is a component of
$G_{w}$ containing $H_{i},$ there are at most $k^{2}-k-1$ vertices in
$V\left(  G_{i}\right)  \backslash V\left(  H_{i}\right)  ;$ since $\left\vert
G_{i}\right\vert >3k^{2},$ the order of $H_{i}$ satisfies
\[
\left\vert H_{i}\right\vert >3k^{2}-k^{2}+k+1>2k^{2}.
\]

Assume for a contradiction that $H_{i}$ is a\ component of $H$ such that
$H_{i}\nsubseteq S_{n_{i},k-1}^{+}.$ Note that $n_{i}>2k^{2}>2k+1,$
$P_{2k+1}\nsubseteq H_{i},$ and $\delta\left(  H_{i}\right)  \geq k-1.$ Using
Theorem \ref{AS+}, we see that $H_{i}$ is one of the graphs from clauses
\emph{(ii), (iii) }or \emph{(iv). }Now Corollary \ref{cAS+} implies that
\[
e\left(  H_{i}\right)  \leq\frac{kn_{i}}{2}.
\]
Since $H$ is $P_{2k+1}$-free, we have
\[
e\left(  H\right)  =e\left(  H_{i}\right)  +e\left(  V\left(  H\right)
\backslash V\left(  H_{i}\right)  \right)  \leq\frac{kn_{i}}{2}+(k-1)\left(
n-1-\left\vert H^{\prime}\right\vert -n_{i}\right)  .
\]
and from (\ref{lb}) we know that%
\[
e\left(  H\right)  \geq\left(  k-1\right)  n-k^{2}+2-\left(  k-2\right)
\left\vert H^{\prime}\right\vert .
\]
After some algebra we see that%
\[
k^{2}-k-1\geq\left(  \frac{1}{2}k-1\right)  n_{i}+\left\vert H^{\prime
}\right\vert \geq k^{2},
\]
a contradiction, showing that $H_{i}\subset S_{n_{i},k-1}^{+},$ and completing
the proof of \ Claim 3.\medskip

Our next goal is to prove that each component $J$ of $G_{w}$ is isomorphic to
$S_{\left\vert J\right\vert ,k-1}^{+}.$ Since $G$ has maximal $q\left(
G\right)  ,$ it is enough to prove that each component $J$ of $G_{w}$
satisfies $J\subset S_{\left\vert J\right\vert ,k-1}^{+}.$

Let\ $J$ be a component of $G_{w}.$ Note that $J$ contains exactly one
component $F$ of $H,$ as otherwise it would consist solely of vertices from
$H^{\prime},$ which are at most $k^{2}-k-1,$ while $J$ has more than $3k^{2}$
vertices. Set $m:=\left\vert F\right\vert ;$ Claim 3 implies that $F\subset
S_{m,k-1}^{+}.$ Write $A$ for the set of $k-1$ dominating vertices of
$S_{m,k-1}^{+};$ let $B:=V\left(  F\right)  \backslash A$ and
\[
C:=V\left(  J\right)  \backslash V\left(  F\right)  .
\]

Since $\delta\left(  F\right)  \geq k-1,$ the bipartite subgraph of $F$
induced by the vertex classes $A$ and $B$ contains at least $\left\vert
A\right\vert \left\vert B\right\vert -2$ edges. If $G\left[  B\right]  $
contains an edge, then each vertex in $B$ is endvertex of a path
$P_{2k}\subset F;$ since $P_{2k+1}\nsubseteq J$, each vertex of $C$ may be
joined only to vertices from $A.$ Therefore, $J\subset S_{\left\vert
J\right\vert ,k-1}^{+}$ as long as $G\left[  B\right]  $ contains an edge.

Assume therefore that the set $B$ is independent. Together with $\delta\left(
F\right)  \geq k-1,$ this assumption implies that $A$ and $B$ induce a
complete bipartite graph in $G.\medskip$

\textbf{Claim 4. \ }\emph{The set }$C$\emph{ is independent.} \medskip

\emph{Proof. }Let
\begin{align*}
C_{A}  &  :=\left\{  u:u\in C\text{ \ \ and \ \ }\Gamma\left(  u\right)  \cap
A\neq\varnothing\right\}  ,\\
C_{B}  &  :=\left\{  u:u\in C\text{ \ \ and \ \ }\Gamma\left(  u\right)  \cap
B\neq\varnothing\right\}  ,\\
C^{\prime}  &  :=C\backslash C_{B}.
\end{align*}
Our main goal is to prove that the set $C^{\prime}$ is independent, which
easily implies that $C$ is independent as well. Assume for a contradiction
that $G\left[  C^{\prime}\right]  $ contains edges. This fact implies that
$C_{B}=\varnothing,$ as $P_{2k+1}\nsubseteq J$. For the same reason we see
that $G\left[  C^{\prime}\right]  $ contains no $P_{4}$ or cycles.

Further, $G\left[  C^{\prime}\right]  $ contains no isolated vertices. Indeed,
if $u\in C^{\prime}$ and $u$ is an isolated vertex in $G\left[  C^{\prime
}\right]  ,$ then it has to be joined to all vertices of $A$ as $q\left(
G\right)  $ is maximal; we see that $u$ has $k-1$ neighbors in $H$\ and so $u$
cannot be removed by the procedure of Claim 2, a contradiction.

Hence $G\left[  C^{\prime}\right]  $ is a disjoint union of edges and stars.
Note that if $S$ is a star in $G\left[  C^{\prime}\right]  $ of order at least
$3,$ then its center belongs to $C_{A},$ but no other vertex of $S$ belongs to
$C_{A},$ as $P_{2k+1}\nsubseteq J.$

Next, assume that $G\left[  C^{\prime}\right]  $ contains a star $S$ of order
$t\geq3,$ such that its center $i$ is joined to exactly one vertex $u\in A;$
let $u_{1},\ldots,u_{t-1}$ be the peripheral vertices of $S$. Remove the edges
$\left\{  u_{1},i\right\}  ,\ldots,\left\{  u_{t-1},i\right\}  $ and add the
edges $\left\{  u_{1},u\right\}  ,\ldots,\left\{  u_{t-1},u\right\}  ;$ write
$G^{\prime}$ for the resulting graph, which obviously satisfies the hypothesis
of Theorem A. We shall show that $q\left(  G^{\prime}\right)  >q\left(
G\right)  .$ First, by symmetry,
\[
x_{u_{1}}=\cdots=x_{u_{t-1}}=p.
\]
Next we have
\begin{align*}
qp  &  =2p+x_{i}+x_{w},\\
qx_{u}  &  \geq\left(  m+1\right)  x_{u}+x_{i}+x_{w},
\end{align*}
and after some algebra we find that $x_{u}>p.$ Also,
\[
qx_{i}=\left(  t+1\right)  x_{i}+\left(  t-1\right)  p+x_{u}+x_{w}<\left(
t+1\right)  x_{i}+tx_{u}+x_{w},
\]
and after some algebra we find that%
\[
\left(  q-m-1+t\right)  x_{u}>\left(  q-t\right)  x_{i},
\]
implying that $x_{u}>x_{i}.$ Now a brief calculation shows that $q\left(
G^{\prime}\right)  >q\left(  G\right)  ,$ contradicting the choice of $G.$

Hence, if $S\subset G\left[  C^{\prime}\right]  $ is a star of order at least
$3,$ then is center is joined to more than one vertex in $A.$

Next assume that $G\left[  C^{\prime}\right]  $ is connected. If $G\left[
C^{\prime}\right]  $ is just one edge, then $J\subset S_{m,k-1}^{+},$ and so
$J=S_{m,k-1}^{+}$ as $q\left(  G\right)  $ is maximal. If $G\left[  C^{\prime
}\right]  $ is a star $S$ of order at least $3,$ then its center $i$ is joined
to more than one vertices $A.$ Since $q\left(  G\right)  $ is maximal, $i$
must be joined to all vertices in $A;$ thus $i$ has $k-1$ neighbors in $H$ and
so $i$ cannot be removed by the procedure of Claim 2, a contradiction. So
$G\left[  C^{\prime}\right]  $ has more than one component.

Finally, assume that $u$ is a vertex of $A$ having a neighbor in $C^{\prime}.$
If $G\left[  C^{\prime}\right]  $ contains a star $S$, then the center of $S$
may be joined only to $u,$ as otherwise we can find a $P_{2k+1}\subset J$
using an additional component of $G\left[  C^{\prime}\right]  .$ Hence
$G\left[  C^{\prime}\right]  $ contains only disjoint edges. Clearly each edge
of $G\left[  C^{\prime}\right]  $ contains a vertex of $C_{A}$ and all such
vertices must be joined exactly to $u$ as $P_{2k+1}\nsubseteq J.$ Since
$q\left(  G\right)  $ is maximal,\ we see that $A$ induces a complete graph,
and both ends of each disjoint edge in $G\left[  C^{\prime}\right]  $ are
joined to $u.$ We shall show that in this case $q\left(  G\right)  $ is not maximal.

Indeed, let $v\in A\backslash\left\{  u\right\}  $ and let $\mathbf{x}=\left(
x_{1},\ldots,x_{n}\right)  $ be a positive unit eigenvector to $q\left(
G\right)  .$ Suppose that $\left\{  i,j\right\}  $ is an isolated edge in
$G\left[  C^{\prime}\right]  .$ Remove $\left\{  i,j\right\}  ,$ add the edges
$\left\{  i,v\right\}  $ and $\left\{  j,v\right\}  ;$ write $G^{\prime}$ for
the resulting graph, which obviously satisfies the hypothesis of Theorem A. By
symmetry, $x_{i}=x_{j};$ note that
\begin{align*}
qx_{v}  &  =mx_{v}+\sum_{s\in A\cup B\backslash\left\{  v\right\}  }%
x_{s}+x_{w}>mx_{v}+x_{u}+x_{w},\\
qx_{i}  &  =3x_{i}+x_{u}+x_{j}+x_{w}=4x_{i}+x_{u}+x_{w}.
\end{align*}
After some algebra we find that
\[
x_{v}>\frac{q-4}{q-m}x_{i}>\frac{q-4}{q-2k^{2}}x_{i}>x_{i},
\]
and a brief calculation shows that $q\left(  G^{\prime}\right)  >q\left(
G\right)  ,$ contradicting the choice of $G.$ This completes the proof that
$C^{\prime}$ is independent. Therefore $C$ is also independent, as no edge in
$C$ can be incident to a vertex in $B$ as $P_{2k+1}\nsubseteq J.$ This
completes the proof of Claim 4.\medskip

Further, we can assume that $C^{\prime}=\varnothing,$ as if $u$ is vertex in
$C^{\prime},$ then, in $G_{w},$ $u$ can be joined only to vertices of $A;$
since $q\left(  G\right)  $ is maximal, $u$ is joined to each vertex in $A;$
thus $i$ has $k-1$ neighbors in $H$ and so $i$ cannot be removed by the
procedure of Claim 2, a contradiction.\medskip

\textbf{Claim 5. }\emph{Either }$J\subset S_{m+1,k-1}^{+}$ \emph{or} \emph{the
set }$C$\emph{ is empty.\medskip}

\emph{Proof. }Observe that if a vertex $u\in C_{B}$ is joined to two or more
vertices from $B,$ then $P_{2k+1}\subset J,$ so each vertex in $C_{B}$ is
joined to exactly one vertex in $B.$ Now if $C_{B}$ has two distinct vertices
that are joined to two distinct vertices in $B,$ then clearly $P_{2k+1}\subset
J.$ Therefore, all vertices in $C_{B}$ are joined to the same vertex of $B,$
say $u\in B$.

Suppose that $C_{A}\neq\varnothing$ and let $v\in C_{A}.$ Clearly
$C_{B}\backslash C_{A}=\varnothing$, as $P_{2k+1}\nsubseteq J;$ therefore
$C_{A}=C_{B}=\left\{  v\right\}  ,$ implying that $J\subset S_{m+1,k-1}^{+}.$

Hence we may assume that $C_{A}=\varnothing,$ that is to say, all vertices in
$C$ are joined only to vertices in $B.$ Now the graph $J\ $looks as follows:
the set $C$ is independent and all vertices of $C$ are joined exactly to the
vertex $u\in B.$ We shall show that in this case $q\left(  G\right)  $ is not maximal.

Indeed, choose a vertex $v\in A$ and let $C=\left\{  u_{1},\ldots
,u_{t}\right\}  ;$ remove the edges $\left\{  u_{1},u\right\}  ,\ldots
,\left\{  u_{t},u\right\}  $ and add the edges $\left\{  u_{1},v\right\}
,\ldots,\left\{  u_{t},v\right\}  ;$ write $G^{\prime}$ for the resulting
graph, which satisfies the hypothesis of Theorem A. We shall show that
$x_{v}>x_{u},$ which obviously implies that $q\left(  G^{\prime}\right)
>q\left(  G\right)  ,$ contradicting the choice of $G.$ Note that by
symmetry,
\[
x_{u_{1}}=\cdots=x_{u_{t}}\text{\ \ and \ \ }x_{s}=x_{v}\text{ for every }s\in
A.
\]
Therefore, letting $x_{u_{1}}=p,$ wee see that
\begin{align*}
qx_{v}  &  >mx_{v}+x_{u}+x_{w},\\
qx_{u}  &  =\left(  k+t\right)  x_{u}+\left(  k-1\right)  x_{v}+tp+x_{w},\\
qp  &  =2p+x_{u}+x_{w}.
\end{align*}
After some algebra we first find that $x_{v}>p,$ and then $x_{v}>x_{u},$ as
claimed. This completes the proof of Claim 5.\medskip

At this stage we see that each component $J$ of\ $G_{w}$ satisfies
$J=S_{\left\vert J\right\vert ,k-1}^{+};$ to finish the proof we must show
that $G_{w}$ has only one component.\emph{ }Assume for a contradiction that
$G_{w}$ contains two components, say $G_{1}=S_{n_{1},k-1}^{+}$ and
$G_{2}=S_{n_{2},k-1}^{+}.$ We shall show that in this case $q\left(  G\right)
$ is not maximal, which contradicts the choice of $G.$

Let $u_{1},\ldots,u_{k-1}\ $and $v_{1},\ldots,v_{k-1}$ be the dominating
vertices in $G_{1}$ and $G_{2}.$ Let $\mathbf{x}=\left(  x_{1},\ldots
,x_{n}\right)  $ be a positive unit eigenvector to $q\left(  G\right)  .$ By
symmetry,
\[
x_{u_{1}}=\cdots=x_{u_{k-1}},\text{ \ \ }x_{v_{1}}=\cdots=x_{v_{k-1}},\text{
\ and \ \ \ }x_{u_{1}}\geq x_{v_{1}}.
\]

Now merge the components $G_{1}$ and $G_{2}$ into one component $F$ by
removing all edges of $G_{2}$ and joining the vertices of $G_{2}$ to each of
the vertices $u_{1},\ldots,u_{k-1}.$ Note that $F$ is isomorphic to
$S_{n_{1}+n_{2},\text{ }k-1}^{+}$ and $u_{1},\ldots,u_{k-1}$ are its
dominating vertices. Writing $G^{\prime}$ for the new graph, we shall show
that $q\left(  G^{\prime}\right)  >q\left(  G\right)  $. Indeed, let
\[
W:=V\left(  S_{n_{2},k-1}^{+}\right)  \backslash\left\{  v_{1},\ldots
,v_{k-1}\right\}
\]
and let $v^{\prime},v^{\prime\prime}\in W$ be the two exceptional vertices of
$G_{2}$ such that $\left\{  v^{\prime},v^{\prime\prime}\right\}  \in E\left(
G\right)  .$ By symmetry, $x_{v^{\prime}}=x_{v^{\prime\prime}}$ and from the
eigenequations for $q$ we see that
\begin{align*}
qx_{v^{\prime}}  &  =(k+1)x_{v^{\prime}}+x_{v^{\prime\prime}}+\left(
k-1\right)  x_{v_{1}}+x_{w}=\left(  k+2\right)  x_{v^{\prime}}+\left(
k-1\right)  x_{v_{1}}+x_{w},\\
qx_{v_{1}}  &  =n_{2}x_{v_{1}}+\left(  k-2\right)  x_{v_{1}}+x_{v^{\prime}%
}+x_{v^{\prime\prime}}+x_{w}+\sum_{s\in W\backslash\left\{  v^{\prime
},v^{\prime\prime}\right\}  }x_{s}\\
&  >\left(  n_{2}+k-2\right)  x_{v_{1}}+2x_{v^{\prime}}+x_{w}.
\end{align*}
Excluding $x_{w}$ from these relations, after some algebra we see that
\[
\left(  q-n_{2}+1\right)  x_{v_{1}}>\left(  q-k\right)  x_{v^{\prime}}%
\]
and so $x_{v^{\prime\prime}}=x_{v^{\prime}}<x_{v_{1}}.$ Further,
\begin{align*}
q\left(  G^{\prime}\right)  -q\left(  G\right)   &  \geq\left\langle Q\left(
G^{\prime}\right)  \mathbf{x},\mathbf{x}\right\rangle -\left\langle Q\left(
G\right)  \mathbf{x},\mathbf{x}\right\rangle \\
&  =\sum_{i\in W}\left(  k-1\right)  \left(  \left(  x_{u_{1}}+x_{i}\right)
^{2}-\left(  x_{v_{1}}+x_{i}\right)  ^{2}\right) \\
&  +\left(  k-1\right)  ^{2}\left(  x_{u_{1}}+x_{v_{1}}\right)  ^{2}-2\left(
k-1\right)  \left(  k-2\right)  x_{v_{1}}^{2}-4x_{v^{\prime}}^{2}\\
&  >4\left(  k-1\right)  ^{2}x_{v_{1}}^{2}-2\left(  k-1\right)  \left(
k-2\right)  x_{v_{1}}^{2}-4x_{v_{1}}^{2}\geq0.
\end{align*}
This contradiction shows that indeed, $G_{w}=S_{n-1,k-1}^{+}$ and so
$G=S_{n,k}^{+}.$ Theorem A is proved and so is Theorem \ref{mt1}.
\end{proof}

\subsection{Proof of Theorem \ref{mt}}

\begin{proof}
Assume for a contradiction that $G$ is a $C_{2k+2}$-free graph of order
$n>400k^{2},$ with $q\left(  G\right)  \geq q\left(  S_{n,k}^{+}\right)  .$ To
prove the theorem we shall show that $G=S_{n,k}^{+}.$ For short, set
$q:=q\left(  G\right)  $ and $V:=V\left(  G\right)  .$

Our proof of Theorem \ref{mt} will be based on a careful analysis of the
Merris bound (\ref{Mb}). Thus, let $w\in V$ be a vertex for which the
expression
\[
d\left(  w\right)  +\frac{1}{d\left(  w\right)  }\sum_{\left\{  w,i\right\}
\in E\left(  G\right)  }d\left(  i\right)
\]
is maximal. First note that $d\left(  w\right)  \geq2k-1,$ as otherwise, using
(\ref{Mb}), we obtain a contradiction
\[
q\left(  G\right)  \leq d\left(  w\right)  +\frac{1}{d\left(  w\right)  }%
\sum_{\left\{  w,i\right\}  \in E\left(  G\right)  }d\left(  i\right)  \leq
d\left(  w\right)  +\Delta\left(  G\right)  \leq n+2k-3<q\left(  S_{n,k}%
^{+}\right)  .
\]

We shall show that $d\left(  w\right)  \geq n-2.$ Indeed, set $A:=\Gamma
\left(  w\right)  $ and $B:=V\left(  G\right)  \backslash\left(  \Gamma\left(
w\right)  \cup\left\{  w\right\}  \right)  .$ Obviously, $\left\vert
A\right\vert =d\left(  w\right)  $ and $\left\vert B\right\vert =n-d\left(
w\right)  -1.$ The assumption $C_{2k+2}\nsubseteq G\ $implies that the graph
$G_{w}$ contains no path $P_{2k+1}$ with both endvertices in $A$. Therefore,
using Lemma \ref{Nl}, we see that
\begin{align*}
d\left(  w\right)  +\frac{1}{d\left(  w\right)  }\sum_{\left\{  w,i\right\}
\in E\left(  G\right)  }d\left(  i\right)   &  =d\left(  w\right)
+1+\frac{2e\left(  A\right)  +e\left(  A,B\right)  }{d\left(  w\right)  }\\
&  \leq d\left(  w\right)  +1+\frac{\left(  2k-1\right)  d\left(  w\right)
+k\left(  n-d\left(  w\right)  -1\right)  }{d\left(  w\right)  }\\
&  =d\left(  w\right)  +k+\frac{k\left(  n-1\right)  }{d\left(  w\right)  }.
\end{align*}
Note that the function $x+\left[  k\left(  n-1\right)  -1\right]  /x$ is
convex for $x>0;$ hence, the maximum of the expression
\[
d\left(  w\right)  +\frac{k\left(  n-1\right)  }{d\left(  w\right)  }%
\]
is attained for the minimum or the maximum admissible values for $d\left(
w\right)  .$ Thus, if
\[
2k-1\leq d\left(  w\right)  \leq n-3,
\]
then
\begin{align*}
q\left(  G\right)   &  \leq d\left(  w\right)  +\frac{k\left(  n-1\right)
}{d\left(  w\right)  }\leq\max\left\{  2k-1+\frac{k\left(  n-1\right)  }%
{2k-1},n-3+\frac{k\left(  n-1\right)  }{n-3}\right\} \\
&  <n+2k-2-\frac{2\left(  k^{2}-k\right)  }{n+2k-3}<q\left(  S_{n,k}%
^{+}\right)  ,
\end{align*}
a contradiction, showing that $d\left(  w\right)  \geq n-2$.

At that stage we are left with two cases: $d\left(  w\right)  =n-1,$ covered
by Theorem \ref{mt1} and $d\left(  w\right)  =n-2,$ which will be disposed of
in the rest of the proof.

Let $v$ be the single vertex of $G$ such that $v\notin\Gamma\left(  w\right)
.$ Let $G^{\prime}$ be the graph obtained by adding the edge $\left\{
w,v\right\}  $ to $G.$ Since $\Delta\left(  G^{\prime}\right)  =n-1,$ and
$q\left(  G^{\prime}\right)  \geq q\left(  G\right)  \geq q\left(  S_{n,k}%
^{+}\right)  ,$ Theorem \ref{mt1} implies that $G^{\prime}$ contains a cycle
$C_{2k+2},$ which obviously contains the edge $\left\{  w,v\right\}  .$ Hence,
$G$ contains a path $P_{2k+2}$ with endvertices $w$ and $v,$ and moreover, $w$
is adjacent to all vertices of this path except $v.$ This is a definite
situation, and it is easy to see that if $d\left(  v\right)  \geq2,$ then
$C_{2k+2}\subset G;$ hence, $d\left(  v\right)  =1$.

Write $u$ for the neighbor of $v$, and let $\mathbf{x}=(x_{1},...,x_{n})$ be
the unit positive eigenvector to $q.$ The eigenequation for the vertex $v$
gives%
\[
qx_{v}=x_{v}+x_{u}%
\]
and so%
\[
x_{v}=\frac{1}{q-1}x_{u}.
\]
Since $d\left(  u\right)  \leq n-2$, there is a vertex $t\in\Gamma\left(
w\right)  \backslash\left(  \Gamma\left(  u\right)  \cup\left\{  u\right\}
\right)  .$ Then, from the eigenequation for $t$ we see that
\[
qx_{t}\geq x_{t}+x_{w},
\]
and so
\[
x_{t}\geq\frac{1}{q-1}x_{w}%
\]
The eigenequation for the vertex $w$ is%
\begin{equation}
qx_{w}=\left(  n-2\right)  x_{w}+\sum_{i\in V\backslash\left\{  w,v\right\}
}x_{i},\text{ } \label{ew}%
\end{equation}
while the eigenequation for the vertex $u$ implies that
\begin{equation}
qx_{u}=d\left(  u\right)  x_{u}+\sum_{i\in\Gamma\left(  u\right)  }x_{i}%
\leq\left(  n-2\right)  x_{u}+\sum_{i\in V\backslash\left\{  u,t\right\}
}x_{i}. \label{eu}%
\end{equation}
Subtracting (\ref{eu}) from (\ref{ew}), we find that%
\[
\left(  q-n+3\right)  \left(  x_{w}-x_{u}\right)  \geq x_{t}-x_{v}\geq\frac
{1}{q-1}x_{w}-\frac{1}{q-1}x_{u}%
\]
and so, $x_{w}\geq x_{u}.$

Let $G^{\prime}$ be the graph obtained from $G$ by removing the edge $\left\{
u,v\right\}  $ and adding the edge $\left\{  w,v\right\}  .$ Comparing the
quadratic forms of $Q\left(  G\right)  $ and $Q\left(  G^{\prime}\right)  ,$
we find that $q\left(  G^{\prime}\right)  \geq q\left(  G\right)  \geq
q\left(  S_{n,k}^{+}\right)  .$ However, $G^{\prime}\neq S_{n,k}^{+}$\ and
$\Delta\left(  G^{\prime}\right)  =n-1;$ hence Theorem \ref{mt1} implies that
$C_{2k+2}\subset G^{\prime},$ and consequently $C_{2k+2}\subset G,$ as no
cycle of $G^{\prime}$contains $v.$ This contradiction completes the proof of
Theorem \ref{mt}.
\end{proof}

\section{Concluding remarks}

Theorem \ref{mt} and the main result of \cite{Yua14},\cite{Yua14a} prove
completely Conjecture \ref{con1}. An important ingredient of our proof,
Theorem \ref{AS+}, which is a nonspectral extremal result, has been obtained
in \cite{NiYu13}. We would like to reiterate a similar, but yet unproven
conjecture for the spectral radius $\mu\left(  G\right)  ,$ raised in
\cite{Nik10c}.

\begin{conjecture}
\label{con2} Let $k\geq2$ and let $G$ be a graph of sufficiently large order
$n.$ If $G$ has no cycle of length $2k+2,$ then $\mu\left(  G\right)
<\mu\left(  S_{n,k}^{+}\right)  ,$ unless $G=S_{n,k}^{+}.$
\end{conjecture}

It is somewhat surprising that Conjecture \ref{con2} turned out to be more
difficult than Conjecture \ref{con1}, given that in general it is easier to
work with the spectra radius than with the $Q$-index of a graph. Finally, let
us note that the corresponding problem about the maximum number of edges in a
$C_{2k}$-free graph of order $n$ is notoriously difficult and is solved only
for very few values of $k.$


\bigskip

\begin{thebibliography}{99}                                                                                               %


\bibitem {Bol98}B. Bollob\'{a}s, \emph{Modern Graph Theory}\textit{,} Graduate
Texts in Mathematics, 184, Springer-Verlag, New York (1998).

\bibitem {Cve10}D. Cvetkovi\'{c}, Spectral theory of graphs based on the
signless Laplacian, Research Report, (2010), available at:
\emph{http://www.mi.sanu.ac.rs/projects/signless\_L\_reportApr11.pdf.}

\bibitem {Das04}K. Das, Maximizing the sum of the squares of the degrees of a
graph, \emph{Discrete Math }\textbf{285} (2004), 57-66.

\bibitem {Dir52}G.A. Dirac, Some theorems on abstract graphs. \emph{Proc.
London Math. Soc.} \textbf{2} (1952) 69-81.

\bibitem {ErGa59}P. Erd\H{o}s and T. Gallai, On maximal paths and circuits of
graphs, \emph{Acta Math. Acad. Sci. Hungar} \textbf{10} (1959), 337--356.

\bibitem {FeYu09}L. Feng and G. Yu, On three conjectures involving the
signless laplacian spectral radius of graphs, \emph{Publ. Inst. Math.
(Beograd) (N.S.) }\textbf{85 }(2009), 35-38.

\bibitem {FNP13}M.A.A. de Freitas, V. Nikiforov, and L. Patuzzi, Maxima of the
$Q$-index: forbidden $4$-cycle and $5$-cycle, \emph{Electronic J. Linear
Algebra }\textbf{26} (2013), 905-916.

\bibitem {Nik09}V. Nikiforov, Degree powers in graphs with forbidden even
cycle, \emph{Electronic J. Combin.} \textbf{15} (2009), R107.

\bibitem {Nik10c}V. Nikiforov, The spectral radius of graphs without paths and
cycles of specified length, \emph{LinearAlgebra Appl.} \textbf{432} (2010), 2243-2256.

\bibitem {NiYu13}V. Nikiforov and X.Y. Yuan, Maxima of the $Q$-index: graphs
without long paths, \emph{Electronic J. Linear Algebra} \textbf{27} (2014), 504-514.

\bibitem {Nik13}V. Nikiforov, An asymptotically tight bound on the $Q$-index
of graphs with forbidden cycles, \emph{Publ. Inst. Math. (Beograd) (N.S.)}
\textbf{95(109)} (2014), 189-199.

\bibitem {Mer98}R. Merris, A note on Laplacian graph eigenvalues, \emph{Linear
Algebra Appl. }\textbf{295 }(1998), 33-35.

\bibitem {Ore60}O. Ore, Note on Hamilton circuits, \emph{Amer. Math. Monthly
}\textbf{67 }(1960), 55.

\bibitem {Yua14}X.Y. Yuan, Maxima of the $Q$-index: forbidden odd cycles,
\emph{Linear Algebra and Appl. }\textbf{458}\emph{\textbf{ }}(2014), 207-216.

\bibitem {Yua14a}X.Y. Yuan, Maxima of the $Q$-index: forbidden odd cycles,
Preprint available at \emph{ArXiv:1401.4363v3}
\end{thebibliography}
\end{document}